\documentclass[12pt]{article}
\usepackage[top=1in, bottom=1in, left=1in, right=1in]{geometry}
\usepackage{graphicx, amsfonts,amssymb,amsthm,amsbsy,latexsym,amscd,amsmath,dsfont,euscript,enumerate,manfnt, marvosym,verbatim, calc,mathrsfs,marvosym,textcomp, color,xcolor,tikz-cd,setspace,stmaryrd,upgreek,bm} 
\usepackage{mathtools}
\usepackage{pgf,tikz,amsmath}
\usepackage{blkarray}
\usepackage{caption}
\usepackage{enumerate}
\usepackage[pagebackref,colorlinks=true,
linkcolor=blue,
urlcolor=red,colorlinks,
citecolor=green]{hyperref}
\usepackage{cleveref}
\usepackage{parskip}
\usepackage{setspace}
\makeatletter
\def\thm@space@setup{%
  \thm@preskip=10pt 
  \thm@postskip=6pt 
}
\makeatother

\usepackage{refcount}

\newtheorem*{theorem*}{Theorem}
\newtheorem{theorem}{Theorem}

\newtheorem{lemma}{Lemma}[section]

\crefname{lemma}{Lemma}{Lemmas}
\crefname{corollary}{Corollary}{Corollaries}

\Crefname{lemma}{Lemma}{Lemmas}
\Crefname{corollary}{Corollary}{Corollaries}

\crefalias{corollary}{corollary}

\crefname{claim}{claim}{claims}
\Crefname{claim}{Claim}{Claims}

\crefname{subclaim}{subclaim}{subclaims}
\Crefname{subclaim}{Subclaim}{Subclaims}

\theoremstyle{definition}

\newtheorem*{acknowledgement}{Acknowledgement}

\usepackage[T1]{fontenc}
\usepackage{lmodern}
\title{An Eventown Result for Multisets}
\author{
    Anirban Banerjee\thanks{\text{anirban.banerjee@iiserkol.ac.in}}
    \qquad
    Abisek Dewan\thanks{\text{ad22rs069@iiserkol.ac.in}}  \\[0.8em]
    \small Department of Mathematics and Statistics, IISER Kolkata, India
}
\date{}

\begin{document}

\maketitle
\begin{abstract}
We study an extension of the Eventown problem of sets to multisets on
$[n] = \{1,2,...,n\}$. A multiset is called
$m$-bounded if each element has multiplicity at most $m$. The intersection of two multisets is defined as the multiset in which each element has multiplicity equal to the minimum of its multiplicities in the two original multisets. A family of $m$-bounded
multisets is called an Eventown family if every member and the intersection of every pair of distinct members has even cardinality, with cardinalities counted with multiplicities.
We determine the maximum size $E_m(n)$ of such a family for
every $m$ and $n$. 
We show that
$
E_{2q}(n)=(q+1)^n,
$ and $
E_{2q+1}(n)=(q+1)^{n-2k}\bigl((q+1)^2+1\bigr)^k,$ where $k=\lfloor n/2\rfloor$.
In particular, the case $m=1$ recovers the classical Eventown
theorem for sets. 
We also characterize all extremal families for $m\ge2$.
For even $m$, the unique extremal family consists of all
multisets in which every element has even multiplicity.
For odd $m\ge3$, every extremal family has a product structure arising from two explicit families on pairs of coordinates, together with an additional coordinate restricted to even multiplicities when $n$ is odd. 
Consequently,
there are exactly $k+1$ extremal types up to permutations
of the ground set.
\end{abstract}



\section{Introduction}

 A family $\mathcal A\subseteq 2^{[n]}$ is called an \emph{Eventown family} if every member has even cardinality and the number of intersections of every two distinct members is even, where $[n]=\{1,\ldots,n\}$. Equivalently, $|A\cap B|$ is even for all $A,B\in\mathcal A$. 
 To answer the question of what the maximum possible size of such a family can be, posed by Erdős, Berlekamp~\cite{berlekamp1969subsets} and Graver\cite{graver1975boolean} independently showed that
\[
|\mathcal A|\le 2^{\lfloor n/2\rfloor}.
\]
Equality is attained by partitioning $2\lfloor n/2\rfloor$ elements into pairs and taking all unions of these pairs. This is also known as classical Eventown problem. On the other hand, the maximum size of an Oddtown family, defined as a collection of odd-sized subsets of $[n]$
 such that every pair of distinct sets has an even-sized intersection, is 
$n$. Thus, changing the parity of the set sizes changes the maximum family size from exponential to linear. The proofs of these results demonstrated the effectiveness of linear algebraic methods in addressing problems in extremal set theory. For an account of these methods and their applications to families of sets with restricted intersection sizes, including modular conditions, see Babai and Frankl~\cite{babai1992linear}. The study of parity-restricted intersections extends beyond ordinary set systems, with analogues of the Oddtown and Eventown problems arising in more general settings,
 such as, vector subspaces over finite field \cite{dey2026q}, permutations\cite{banerjee2026even,lindzey2026eventown}, and lattice paths\cite{shankar2026oddtown}. This  naturally leads to the investigation of parity-restricted intersection problems in other combinatorial frameworks. Here, we study an Eventown problem for multisets, that is, collections in which elements may occur with higher multiplicity.

Multiset analogues of intersection theorems have been studied in several settings. Füredi, Gerbner, and Vizer~\cite{furedi2015discrete} and Meagher and Purdy~\cite{Meagher_Purdy_2011} investigated intersection restrictions for families of fixed-cardinality multisets. More recently, Liao, Lv, Cao, and Lu~\cite{liao2024erdHos} established an Erdős--Ko--Rado theorem for multisets with bounded multiplicities. These results concern nonempty intersections or intersections of a prescribed minimum size. In this paper, we consider a parity restriction instead: intersections must have even cardinality, while the total cardinalities of the members are allowed to vary.

Once the maximum size of families of a given type has been determined, a natural structural question is to characterize those families that attain the corresponding extremal bound.
In the classical Eventown problem, the standard construction using disjoint pairs is not unique, and there exist many non-isomorphic extremal families; see Section~1.1 of \cite{sudakov2018two}. 
Sudakov and Vieira~\cite{sudakov2018two} investigated this structural question for $r$-wise Eventown families of subsets of $[n]$, in which the intersection of every $r$ distinct members has even cardinality. For $r\ge3$ and $n$ sufficiently large depending on $r$, they characterized the extremal families and established a corresponding stability result. 
Now, in this paper, we have characterized the extremal Eventown families of multisets in which at least one element occurs with multiplicity greater than one. The Eventown family of multisets in which each element occurs with multiplicity at most one is just a classical Eventown family.

\subsection{The Multiset Eventown Problem}
Fix positive integers  $m$ and $n$. A \emph{multiset} on $[n]$ is a collection of elements of $[n]$ in which repetitions are allowed. The number of occurrences of an element is called its \emph{multiplicity}. A multiset is called \emph{$m$-bounded} if every element has multiplicity at most $m$. Every $m$-bounded multiset $x$ on $[n]$ corresponds uniquely to a multiplicity vector
\[
x \equiv (x_1,\ldots,x_n)\in\{0,1,\ldots,m\}^n,
\]
where $x_i$ is the multiplicity of the element $i$ in $x$. Throughout this paper, we identify $m$-bounded multisets with their multiplicity vectors.

For two multisets $x$ and $y$, their intersection $x \cap y$
is the multiset in which each element $i$ has multiplicity $\min\{x_i,y_i\}$. The cardinality of their intersection, counting multiplicities, is therefore
\[
I(x,y)=  |x \cap y| = \sum_{i=1}^{n}\min\{x_i,y_i\}.
\]
Thus, the intersection counts the copies shared by the two multisets, rather than merely the elements appearing in both supports. We call a family of multisets
\[
\mathcal F\subseteq\{0,1,\ldots,m\}^n
\]
an  \emph{Eventown family} if $I(x,y)$ is even for every $x,y\in\mathcal F$. 
In particular, each member has an even total multiplicity. Let $E_m(n)$ denote the maximum possible cardinality of such a family.

This extends the classical Eventown problem by taking $m=1$. However, allowing multiplicities introduces an additional difficulty. For ordinary sets, the parity of the intersection cardinality is determined by the scalar product of their characteristic vectors over $\mathbb F_2$. For multisets, in contrast, the parities of the individual multiplicities do not, in general, determine the parity of the intersection cardinality, since the latter depends on the minimum of the corresponding multiplicities. Consequently, one cannot simply reduce the multiplicity vectors modulo $2$ and apply the orthogonality arguments used in the  proof of the classical Eventown theorem.

A trivial upper bound for $E_m(n)$ is 
 $2^{\lfloor mn/2\rfloor}$. To obtain this bound, represent each $m$-bounded multiset as a subset of $[n]\times[m]$, where a multiplicity $x_i$ is represented by selecting the first $x_i$ elements of the block $\{i\}\times[m]$. Under this representation, the cardinality of the intersection of two multisets equals the cardinality of the intersection of their corresponding subsets. The classical Eventown theorem therefore yields the stated upper bound. However, this bound is not sharp in general.
 
In this paper, we determine $E_m(n)$ for every $m$ and $n$, and characterize the extremal families for $m\ge2$.  We prove that the family consisting of all multiplicity vectors with even coordinates is the unique extremal family for every even $m$. Whereas for odd $m\ge3$ and $n\ge2$, there are more than one extremal family. We  characterize all the extremal families for every odd $m\ge3$. 

\subsection{Main Results}
Our main result determines the exact value of $E_m(n)$ for all positive integers $m$ and $n$.
\begin{theorem}\label{Theorem1}
    Let $E_m(n)$ be the maximum size of an Eventown family of $m$-bounded multisets on $[n]$. Then
    \[E_m(n)=
\begin{cases}
(q+1)^n,
& m=2q,\\[2mm]
\bigl((q+1)^2+1\bigr)^{n/2},
& m=2q+1,\quad n\text{ even},\\[2mm]
(q+1)\bigl((q+1)^2+1\bigr)^{(n-1)/2},
& m=2q+1,\quad n\text{ odd}.
\end{cases}\]
\end{theorem}

In particular, $E_1(n)=2^{\lfloor n/2\rfloor}$, recovering the
classical Eventown bound.

 We also characterize all families attaining the above bound. For even
$m=2q$, we prove that the family $\{0,2,\ldots,2q\}^n$ is the unique extremal family. For odd $m=2q+1\ge 3$, we define
\[\mathcal A=\{(0,0)\}\cup\{1,3,\ldots,m\}^2,
\qquad
\mathcal B=\{(m,m)\}\cup\{0,2,\ldots,2q\}^2,\]
and show that up to a permutation of coordinates the extremal families are precisely the families,
\[\mathcal A^t\times\mathcal B^{k-t}
\times\{0,2,\ldots,2q\}^{\,n-2k},\]
for some integer $0\le t\le k$, where $k=\lfloor n/2\rfloor$. Here, the final factor is omitted when $n$ is even.
The case, $m=1$ is exceptional. 
Identifying sets with their characteristic vectors, the extremal families are precisely the self-orthogonal subspaces of $\mathbb{F}_2^n$ of dimension $\lfloor n/2\rfloor$. Therefore, such extremal families need not admit the paired form described above.

\subsection{Outline of the Proof}
Given an Eventown family $\mathcal F$ of $m$-bounded multisets, we successively transform the multiplicities of one coordinate at a time, preserving the Eventown property at each step. For even $m$, the transformation decreases the size of the family by at most a factor of $r=q+1$ at each step. This gives us the desired upper bound. When $m$ is odd, we transform the family $\mathcal F$ into an Eventown family $\mathcal G$ without decreasing its size, such that each coordinate takes values from one of two prescribed sets. We then establish the upper bound using parity vectors, linear algebra, and counting arguments.
 
We next characterize the cases of equality. For even $m$, we prove that every extremal family consists precisely of all $m$-bounded multisets in which every element has even multiplicity. For odd $m\ge 3$, we show that $\mathcal G$ has a product structure determined by $\mathcal A$- and $\mathcal B$-pairs, with an additional coordinate restricted to even multiplicities when $n$ is odd. Finally, equality forces all intermediate families to have maximum size and a backward induction then shows that each transformation leaves the family unchanged, and hence $\mathcal F=\mathcal G$.

\section{Proof of Theorem 1}

We begin with an elementary linear-algebraic observation that will be used in the subsequent counting argument.
\begin{lemma}\label{lem:two-groups} 
Let $C\le \mathbb{F}_2^n$ be a $d$-dimensional self-orthogonal subspace, so that,
\[c \cdot c'=0 \ \ \ \ \ \ \ (c,c'\in C).\]
Then there exist disjoint coordinate sets $I,K\subseteq[n]$, each of cardinality $d$, such that the coordinate restriction maps to either $I$ or $K$ are injective on $C$. In particular, the coordinates indexed by either $I$ or $K$ uniquely determine every vector in $C$.
\end{lemma}

\begin{proof}
If $d=0$, take $I=K=\varnothing$. Assume $d>0$, and let $G$
be a $d\times n$ matrix whose rows form a basis of $C$. For $J\subseteq[n]$, let $G_J$ denote the submatrix of $G$
consisting of the columns indexed by $J$.
Choose $d$ independent columns, indexed by $I$.
Since $G_I$ is invertible, the restriction or projection map
\[\pi_I:C\longrightarrow \mathbb{F}_2^I\]
is a bijection. Here, $\mathbb{F}_2^I$ denotes the space of binary vectors indexed by $I$.
We claim that restriction to the complementary coordinates is
also injective. Indeed, suppose $z\in C$ is supported on $I$, that is, $z_i=0$ for all $i\notin I$. By the self-orthogonality of $C$, we have
\[
0=z\cdot c=\sum_{i\in I}z_i c_i,
\qquad\text{for every }c\in C.
\]
Since $c|_I$ ranges over all of $\mathbb{F}_2^I$, this forces
$z|_I=0$, and hence $z=0$.

Therefore restriction to $I^c=[n]\setminus I$ is injective, and thus
 $G_{I^c}$ has rank $d$.
Choosing $d$ independent columns of this matrix gives a set
$K\subseteq I^c$, of size $d$, such that restriction to $K$
is also injective.
\end{proof}

We next examine the multiplicities that can occur in a fixed coordinate of an Eventown family. We fix the values in all other coordinates and record which even and odd values can appear in the chosen coordinate. To state the next lemma, we introduce the following notation.

For a family $\mathcal F\subseteq\{0,1,\ldots,m\}^n$,
a coordinate $j\in[n]$, and a tuple
$u=(u_1,\ldots,u_{n-1})\in\{0,1,\ldots,m\}^{n-1}$, let
\[
\iota_j(s,u)
=
(u_1,\ldots,u_{j-1},s,u_j,\ldots,u_{n-1})
\]
denote the vector obtained by inserting $s$ into coordinate $j$.
Define
\[
\begin{aligned}
E_j(u)
&=\{s\in\{0,1,\ldots,m\}:
      \iota_j(s,u)\in\mathcal F,\ s\text{ is even}\},\\
O_j(u)
&=\{s\in\{0,1,\ldots,m\}:
      \iota_j(s,u)\in\mathcal F,\ s\text{ is odd}\}.
\end{aligned}
\]
Take
\[
a_j=\max_u |E_j(u)|,
\qquad
b_j=\max_u |O_j(u)|,
\]
where both maxima are taken over
$u\in\{0,1,\ldots,m\}^{n-1}$. When $j=1$, we write $(s,u)$ instead of $\iota_1(s,u)$ and omit the coordinate subscript, writing $E(u), O(u), a,$ and $b$ in place of $E_1(u), O_1(u), a_1,$ and $b_1$, respectively.

\begin{lemma}
\label{ordering lemma}
Let $m=2q+\delta$, where $\delta\in\{0,1\}$, and put $r=q+1$.
Let $\mathcal F\subseteq\{0,1,\ldots,m\}^n$ be an Eventown family.
Then, for every coordinate $j\in[n]$,
\[
a_j+b_j\le r+\delta.
\]
\end{lemma}
\begin{proof}
    If $a_j=0$ or $b_j=0$, the conclusion follows from the fact that $\{0,1,...,m\}$ contains $r$ even values and $q+\delta \le r$ odd values.

    Otherwise, choose $u,v$ such that $|E_j(u)|=a_j$ and $|O_j(v)|=b_j$. Let $E_j(u)=\{s_1<s_2<...<s_{a_j}\}$ and $O_j(v)=\{t_1<t_2<...<t_{b_j}\}$. Then, we have 
    \[\min(s_i,t_{i'})+I(u,v) \equiv 0 \pmod2\ \ \ \ \forall \ i\in[a_j],\ i'\in[b_j].\]
    Thus all $\min(s_i,t_{i'})$ have same parity. Since $s_i$'s are even and $t_{i'}$'s are odd, either $s_1<...<s_{a_j}<t_1<...<t_{b_j}$ or $t_1<...<t_{b_j}<s_1<...<s_{a_j}$. So, either every element of $E_j(u)$ lies below every element of $O_j(v)$, or conversely.

    In the first case, write $\max E_j(u)=2l$. There are at most $l+1$ even values at most $2l$, and $q+\delta -l$ odd values greater than $2l$. Hence
    \[a_j \le l+1,\ \ \ \ \ \ b_j \le q+\delta-l,\]
    so $a_j+b_j\le r+\delta.$

    In the second case, write $\max O_j(v)=2l-1$. Then,
    \[b_j \le l,\ \ \ \ \ \ \ \ \ a_j \le q-l+1,\]
    and therefore $a_j+b_j \le r\le r+\delta$. This proves the lemma.
\end{proof}
We next study Eventown families whose multiplicities in each
coordinate lie in one of two sets defined below. The following
lemma gives a bound for such families, which we will use to bound
the size of the original Eventown family $\mathcal F$. It also
describes all families attaining this bound when $m\ge3$.

Let $m=2q+1$. Define
\[
V_e=\{0,2,\ldots,2q\},
\qquad
V_o=\{1,3,\ldots,m\}.
\]
Set
\[
Q_e=V_e\cup\{m\},
\qquad
Q_o=\{0\}\cup V_o,
\]
and define the two families of pairs
\[
\mathcal A=V_o^2\cup\{(0,0)\},
\qquad
\mathcal B=V_e^2\cup\{(m,m)\}.
\]

\begin{lemma}
\label{lem:counting-G}
Let $\mathcal G\subseteq\{0,1,\ldots,m\}^n$ be an Eventown family of multisets
such that, in each coordinate, the multiplicities lie in either $Q_e$ or $Q_o$, with the choice fixed for that coordinate. Then
\[
|\mathcal G|
\le r^{n-2k}(1+r^2)^k.
\]
Moreover, if $r>1$, equivalently $m\ge3$, equality holds if and
only if, up to a permutation of the coordinates,
\[
\mathcal G
=
\mathcal A^{\,t}\times
\mathcal B^{\,k-t}\times
V_e^{\,n-2k}
\qquad\text{for some }t\in\{0,\ldots,k\}.
\]
\end{lemma}

\begin{proof}
Observe that the identity $\min(s,t)\bmod2=(s\bmod2)(t\bmod2)\label{1}$ holds within each of $Q_e,Q_o$. For $x \in \mathcal{G}$, define its parity vector as 
\[\pi(x)=(x_1 \bmod2,x_2 \bmod2,...,x_n \bmod2)\in\mathbb{F}_2^n.\]
Using the above identity we get for $x,y\in \mathcal{G}$,
\[\pi(x) \cdot \pi(y)=\sum_{i=1}^n\min(x_i,y_i)=0\ \ \ \ \ \text{in}\ \mathbb{F}_2.\]
Let $C=\text{span}_{\mathbb{F}_2}\{\pi(x):x \in \mathcal{G}\}$. Then $C$ is a self-orthogonal subspace of $\mathbb{F}_2^n$. Therefore if dimension of $C$ is $d$, then $d\le k$.

By Lemma~\ref{lem:two-groups}, there exist disjoint coordinate sets $I,J\subseteq[n]$, each of size $d$, such that either restriction determines a vector of $C$. Put $L=[n]\setminus(I\cup J)$, so $|L|=n-2d$. Note that these restrictions determine the parity vector, not the actual multiplicities.

In each coordinate, the allowed multiplicities have parity counts $(r,1)$ or $(1,r)$. Thus, one parity has $r$ multiplicities, and the other has one. Let $P$ be the set of parity vectors occurring in $\mathcal{G}$. For each $c \in P$, define
\begin{align*}
    a_c &\coloneqq\  \text{the number of allowed multiplicity assignments on} \ I \ \text{having parity pattern }\ c|_I, \\  b_c &\coloneqq\  \text{the number of allowed multiplicity assignments on} \ J \ \text{having parity pattern }\ c|_J.
\end{align*}
Once $c$ is fixed, each coordinate in $L$ has at most $r$ multiplicity choices. Hence the number of members of $\mathcal{G}$ with parity vector $c$ is at most $a_cb_cr^{n-2d}$. Consequently,
\[|\mathcal{G}|\le r^{n-2d}\sum_{c \in P} a_cb_c.\]

Next, we shall bound $\sum_{c \in P} a_c^2$. For every binary pattern $u$ on $I$, let $A(u)$ be the number of allowed multiplicity assignments with that pattern. Since $I$ contains $d$ coordinates, there are $2^d$ possible patterns for $u$. If $u$ selects the parity with $r$ choices in exactly $p$ coordinates, then
\[A(u)=r^{p}.\]
Since $l'$ coordinates are chosen from $d$ coordinates, there are exactly $\binom{d}{p}$ such patterns. Consequently,
\[
\sum_{u\in\{0,1\}^{I}}A(u)^2
=\sum_{p=0}^{d}\binom {d}{p} r^{2p}
=(1+r^2)^d.
\]
Since $I$ determines the complete parity vector, distinct $c \in P$ provide distinct restrictions $c_I$. Therefore, we have,
\[\sum_{c \in P} a_c^2 \le(1+r^2)^d.\]
 Similarly from  $J$ we have 
 \[\sum_{c \in P} b_c^2 \le(1+r^2)^d.\]
Hence
\begin{align*}
    |\mathcal{G}|&\le r^{n-2d}\sum_{c \in P}a_cb_c \\ &\le r^{n-2d} \frac{1}{2}(\sum_{c \in P} a_c^2+\sum_{c \in P} b_c^2) \\ &\le r^{n-2d} (1+r^2)^d \\ & \le r^{n-2k} (1+r^2)^k\ \ \ \ \ \ \ \ \ \ \ [\text{as}\ r\ge1]
\end{align*}
Now we characterize the extremal families. Since $|\mathcal{G}|=r^{n-2k}(1+r^2)^k$, every inequality above is an equality. Consequently, we have the following
\begin{enumerate}
    \item $d=k$,
    \item $P=C$, and 
    \item $a_c=b_c$ for all $c \in C$.
\end{enumerate}
Also, if a parity pattern occurs in $P$, then every allowed choice of multiplicity vector realizing that pattern belongs to $\mathcal{G}$. 

If $n$ is odd, there is one coordinate outside $I \cup J$. This coordinate must contribute exactly $r$ choices for each parity pattern that occurs. Since $0\in C=P$, the even parity must provide $r$ choices. Hence, the multiplicity assigned to this coordinate must belong to $Q_e$. In contrast, the odd parity provides only one choice, so no member of the family can have odd multiplicity in this coordinate. Consequently, the remaining coordinate takes values precisely from $V_e$, independently of the other coordinates.

Now since $I,J$ both determines $C$, their restrictions have the form
\[\{(u,Tu):u \in \mathbb{F}_2^k\},\]
where $T=\pi_J \circ \pi_I^{-1}$ is an invertible linear map. Define $\alpha \in \mathbb{F}_2^{|I|} $ as follows
\[\alpha_i=\begin{cases}
0,
& \text{the $i$-th coordinate of $I$ uses $Q_e$},\\[2mm]
1,
& \text{the $i$-th coordinate of $I$ uses $Q_o$}.
\end{cases}\]
Similarly, define $\beta\in \mathbb{F}_2^{|J|}$ on $J$. Thus, a coordinate of $I$ gives $r$ choices precisely when $u_i=\alpha_i$. Therefore,
\[A(u)=r^{k-\text{wt}(u+\alpha)}.\]
Similarly,
\[B(Tu)=r^{k-\text{wt}(Tu+\beta)},\]
where $B(Tu)$ is the analogous multiplicity assignment count on $J$. By using $a_c=b_c$ for $c\in C$ and $r>1$, we get,
 \[\text{wt}(u+\alpha)=\text{wt}(Tu+\beta)\ \ \ \ \ \ \text{for every}\  u.\]
Taking $u=\alpha$ we have $T\alpha=\beta$. Substituting $u=v+\alpha$ into the preceding equation, we obtain
\[
\operatorname{wt}(Tv)=\operatorname{wt}(v)
\qquad\text{for every }v.
\]
Each standard basis vector has Hamming weight $1$. Since $T$ preserves Hamming weight, it maps every standard basis vector to a standard basis vector. The invertibility of $T$ ensures that these images are distinct. Hence, $T$ is a coordinate permutation, inducing a pairing between the coordinates of $I$ and those of $J$. Within each pair, the two coordinates have the same parity, and both take values exclusively from either $Q_e$ or $Q_o$.

A pair of $Q_e$-coordinates, therefore, gives
\[V_e^2 \cup \{(m,m)\}=\mathcal{B}\]
and a pair of $Q_o$-coordinates gives
\[\{(0,0)\}\cup V_o^2=\mathcal{A}.\]
Since all parity patterns and all permitted assignments occur, we have,
\[\mathcal{G}=\mathcal{A}^t\times\mathcal{B}^{k-t}\times V_e^{n-2k},\]
after a coordinate permutation. The converse is immediate.
\end{proof}

Now, we can prove Theorem \ref{Theorem1}:
 
\begin{proof}[\textbf{Proof of Theorem} \ref{Theorem1}]
For a fixed coordinate $i$ and a family $\mathcal{F}$, define
\begin{align}
    T_e^i&=\{x:\iota_i(s,x)\in \mathcal{F},\text{for some even}\ s\in \{0,1,...,m\}\} \\ T_o^i&=\{y:\iota_i(t,y)\in \mathcal{F},\text{for some odd}\ t\in \{0,1,...,m\}\}
\end{align}
For $i=1$, we drop the superscript and just write $T_e,T_o$.

\textbf{Case I: m is even}

Suppose $\delta=0$, so $m=2q$. If $q=0$, the conclusion is immediate, otherwise, by Lemma \ref{ordering lemma}, we have,
\[
|\mathcal F|\le r\max(|T_e|,|T_o|).
\]
Now, we fix the first coordinate and replace $\mathcal F$ by
\[
\mathcal F^{1}=
\begin{cases}
\{(0,x):x\in T_e\},& if \ |T_e|\ge |T_o|,\\[1mm]
\{(1,y):y\in T_o\},& if\ |T_o|>|T_e|.
\end{cases}
\]
This replacement is clearly an Eventown family.
Moreover,
\[
|\mathcal F^1|=\max(|T_e|,|T_o|)\ge\frac{|\mathcal F|}{r}.
\]
The selected first coordinate is now fixed, and all other coordinates remain unchanged.

Starting with $\mathcal F$, we perform this operation successively on
all $n$ coordinates. 
The resulting family $\mathcal G$ contains at most one member, and at each iteration, the size of the family decreases by a factor of at most $r$.
Consequently,
\[
|\mathcal F|\le r^n |\mathcal{G}|\le r^n=(q+1)^n.
\]
This proves the theorem when $m$ is even.

\textbf{Case II: m is odd}

Suppose $\delta=1$, so $m=2q+1$. Define
\[Q_e=\{0,2,...,2q\}\cup\{m\},\ \ \ \ \ \ Q_o=\{0\}\cup\{1,3,...,2q+1\}\]
both of size $r+1$.

We first transform the family $\mathcal{F}$ into another Eventown family $\mathcal{G}$, without decreasing its size, such that the value in each coordinate is either from $Q_e$ or $Q_o$.

Fix the first coordinate. First suppose that,
\[|\mathcal{F}|\le r \max(|T_e|,|T_0|).\]
If $|T_e|\ge |T_o|$, replace $\mathcal{F}$ by $\{(s,x):s\in\{0,2,...,2q\},x\in T_e\}=V_e \times T_e$; otherwise replace it by $\{(t,y):t\in\{1,3,...,2q+1\},y\in T_o\}=V_o \times T_o$. The replacement family is also an Eventown family, where the chosen first coordinate of each member
lies in $Q_e$ or $Q_o$.
The  size of the new family is $r \max(T_e,T_o)\ge |\mathcal{F}|$.  

Now, suppose $|\mathcal{F}|>r \max(|T_e|,|T_o|)$. This forces $a+b >r$. Therefore, by Lemma \ref{ordering lemma}, we get,
\[a+b=r+1\]
and for every $x\in T_e$ with $a$ number of even extensions and any $y \in T_o$ with $b$ number of odd extensions, the even extensions come before the odd extensions. Therefore, selected tuples have exactly the following extensions:
\begin{align*}
    x&:0,2,...,2l\\ y&:2l+1,2l+3,...,2q+1
\end{align*}
for some $l \in \{0\}\cup[q]$.

Now, every extension of a tuple in $T_e$ precedes every extension of $y$. Consequently, each tuple in $T_e$ has  extensions
at most $2l$.
Similarly, a tuple in $T_o$ has all its extensions above  $2l$.  
Consequently, every even extension precedes every odd extension in  $\mathcal{F}$.

Now, if $|T_e|\ge |T_o|$ we replace $\mathcal{F}$ by \[\mathcal{F}_1^{(1)}=\{(s,x):s\in\{0,2,...,2q\},x\in T_e\} \cup \{(2q+1,y):y \in T_o\},\] otherwise by \[\mathcal{F}_2^{(1)}=\{(0,x):x \in T_e\} \cup \{(t,y):t\in\{1,3,...,2q+1\},y\in T_o\}.\] Both of the families, $\mathcal{F}_1^{(1)}$ and $\mathcal{F}_2^{(1)}$ are Eventown families with the sizes $r|T_e|+|T_o|$ and $|T_e|+r|T_o|$, respectively.

So, when $|T_e|\ge |T_o|$ we have
\[r|T_e|+|T_o|-(a|T_e|+b|T_o|)=(b-1)(|T_e|-|T_o|)\ge0,\]
and for  $|T_o|\ge|T_e|$, we have
\[|T_e|+r|T_o|-(aT_e+b|T_o|)=(a-1)(|T_o|-|T_e|)\ge0.\]
Thus the chosen replacement family, has size at least $a|T_e|+b|T_o|\ge|\mathcal F|$,
preserves the Eventown condition and restricts the first coordinate
to $Q_e$ or $Q_o$.

We apply the above procedure successively to all coordinates, with each iteration modifying only the chosen coordinate of the family members. After all $n$ iterations, we obtain an Eventown family $\mathcal G$ satisfying
\[
|\mathcal F|\le |\mathcal G|,
\]

such that, in each coordinate, every member of $\mathcal G$ takes values exclusively from either $Q_e$ or $Q_o$.
Hence by Lemma \ref{lem:counting-G}, we have,
\[|\mathcal{F}|\le |\mathcal G|
\le r^{n-2k}(1+r^2)^k.\]
Therefore,
\[|\mathcal{F}|\le\begin{cases}
\bigl((q+1)^2+1\bigr)^{n/2},  \qquad\ \ \ \ \ \ \ \ \ \ \ \ \ \text{ for } n\text{ even},\\[2mm]
(q+1)\bigl((q+1)^2+1\bigr)^{(n-1)/2},\qquad \text{ when }n\text{ is odd}.
\end{cases}\]
\end{proof}
\section{Characterization of Extremal Families for \texorpdfstring{$m\ge 2$}{m >= 2}}
Let $\mathcal F^0=\mathcal F$, and let $\mathcal F^i$ denote the family after $i$ iterations of the reduction procedure. At the $i$-th iteration, define
\[
\begin{aligned}
T_e^i
&=\{u:\iota_i(s,u)\in\mathcal F^{i-1}
\text{ for some }s\in V_e\},\\
T_o^i
&=\{u:\iota_i(s,u)\in\mathcal F^{i-1}
\text{ for some }s\in V_o\},\\
a_i
&=\max_u\left|\{s\in V_0:
\iota_i(s,u)\in\mathcal F^{i-1}\}\right|,\\
b_i
&=\max_u\left|\{s\in V_1:
\iota_i(s,u)\in\mathcal F^{i-1}\}\right|,
\end{aligned}
\]
where $u$ ranges over $\{0,1,\ldots,m\}^{n-1}$. Thus, $a_i$ and $b_i$ are the maximum numbers of even and odd extensions, respectively, in coordinate $i$.
\subsection{Case of Even \texorpdfstring{$m$}{m}}
Let $m=2q$. The case $q=0$ is immediate. Suppose $|\mathcal{F}|=r^n$. Now we recall Case I in the proof of Theorem \ref{Theorem1}.
Writing $\mathcal{F}^{i}$ for the family after $i$ iterations, equality forces $|\mathcal{F}^{i}|=r^{n-i}$. Thus, at each iteration the size of the family decreases by a factor of $r$. Since $a_i+b_i\le r$ and $b_i \le r-1$, the condition $|T_o^i|>|T_e^i|$ gives us
\[|\mathcal{F}^{i-1}| \le a_i|T_e^i|+b_i|T_o^i| <r|T_o^i|\]
which contradicts the equality, $|\mathcal F^{i-1}|=r|T_o^i| $.
Thus, $|T_o^i|$ can not be greater than $|T_e^i|$.
So, every chosen constant is $0$ and therefore $\mathcal{F}^{n}=\{(0,0,...,0)\}$.

Now, we recover the original family $\mathcal{F}$. Suppose $|T_e^i|=|T_o^i|$ at the step $i$. Consider an Eventown family 
\[
\mathcal K=\{\iota_i(1,u):u\in T_o^i\}.\] Clearly  $|\mathcal K|=r^{n-i}.$ 
We keep its first $i$ coordinates fixed and apply the same procedure to the remaining coordinates. At each remaining iteration, the size of the family decreases exactly by a factor $r$, and therefore chooses $0$. The member of the final singleton family consequently has exactly one coordinate equal to $1$, contradicting the Eventown condition and thus $|T_e^i|\ne |T_o^i|$. 
Hence $|T_e^i|>|T_o^i|$ at every step $i$.
So we have
\[
r|T_e^i|=|\mathcal F^{i-1}|
\le a_i|T_e^i|+b_i|T_o^i|
\le r|T_e^i|-b_i(|T_e^i|-|T_o^i|),
\] which implies $b_i=0$. 
Since there are $|T_e^i|$ tuples and $|\mathcal{F}^{i-1}|=r|T_e^i|$ member in, each tuple must have all $r$ even extension. Thus each iteration fixes a coordinate that could take any value in $V_e$. Reversing the iteration process, from $\mathcal F^n=\{0\}^n$ we obtain
\[
\mathcal F=V_e^n.
\] Therefore $V_e^n$ is the unique extremal Eventown family when $m$ is even.

\subsection{Case of Odd m}

Let $M=r^{n-2k}(r^2+1)^k.$ Let us consider the following sequence of families obtained from the iteration process described in Case II in the proof of the Theorem \ref{Theorem1}
\[\mathcal F=\mathcal F^0,\mathcal F^1,\ldots,\mathcal F^n=\mathcal G.\]
Since $|\mathcal F|=M$,  we have
\[|\mathcal F^i|=M
\qquad (0\le i\le n).\]
Now, we proceed by backward induction to show that every intermediate family coincides with $\mathcal G$. Let us fix $i$ and consider $\mathcal F^i=\mathcal G$.

Recalling the fact established in Case II in the proof of the Theorem \ref{Theorem1} and Lemma \ref{ordering lemma} we have
\[M\le a_i|T_e^i|+b_i|T_o^i|,
\qquad
a_i+b_i\le r+1.\]
Moreover, if $a_i+b_i=r+1$, then every even extension precedes every odd extension in coordinate $i$.
Also recall that, by the characterization of $\mathcal G$, its coordinates can be grouped into $\lfloor n/2\rfloor$ pairs, with both coordinates in each pair using the same allowed set, either $Q_e$ or $Q_o$. When $n$ is odd, the remaining unpaired coordinate takes every value in $V_0$, independently of the other coordinates.

\medskip
\noindent
\textbf{Subcase I. Coordinate $i$ belongs to a pair in $\mathcal G$.}

Both even and odd multiplicities occur in coordinate $i$, and the iteration producing $\mathcal F^i$ retains both corresponding groups, $T_e^i$ and $T_o^i$.

Now, if the coordinate $i$ belongs to a $\mathcal B$-pair, then
\[\mathcal F^i
=
\{\iota_i(s,u):s\in V_0,\ u\in T_e^i\}
\cup
\{\iota_i(m,v):v\in T_o^i\}.\]
Let $L$ denote the number of multiplicity assignments on the other $n-2$
coordinates. Since
\[\mathcal B=V_0^2\cup\{(m,m)\},\]
we have
\[|T_e^i|=rL,
\qquad
|T_o^i|=L.\]
Consequently,
\[M
\le a_i|T_e^i|+b_i|T_o^i|
\le r|T_e^i|+|T_o^i|
=M.\]
Since $a_i+b_i=r+1$, the equality gives us
\[(b_i-1)\bigl(|T_e^i|-|T_o^i|\bigr)=0.\]
Again, since $r>1$ and $|T_e^i|=r|T_o^i|$, we have $b_i=1$ and $a_i=r$.

Thus some tuple in $T_e^i$ has all even extensions, including
$2q$. The ordering property forces every odd extension in
coordinate $i$ to be $m$. Hence
\[\mathcal F^{i-1}
\subseteq
\{\iota_i(s,u):s\in V_0,\ u\in T_e^i\}
\cup
\{\iota_i(m,v):v\in T_o^i\}
=
\mathcal F^i.\]
Since both families have size $M$, they are equal. So $F^{i-1} = \mathcal{G}$.

Now we consider when the coordinate $i$ belongs to an $\mathcal A$-pair. Here,
\[\mathcal F^i
=
\{\iota_i(0,u):u\in T_e^i\}
\cup
\{\iota_i(t,v):t\in V_1,\ v\in T_o^i\}.\] and 
\[|T_e^i|=L,
\qquad
|T_o^i|=rL.\]
Now, similarly, we have $a_i=1$ and $b_i=r$.

Some tuples in $T_o^i$ have all odd extensions, including $1$. This forces every even extension in coordinate
$i$ to be $0$. Consequently,
\[\mathcal F^{i-1}
\subseteq
\{\iota_i(0,u):u\in T_e^i\}
\cup
\{\iota_i(t,v):t\in V_1,\ v\in T_o^i\}
=
\mathcal F^i.\]
Equal cardinalities again imply equality and hence $\mathcal{F}^{i-1} = \mathcal{G}$.

\medskip
\noindent
\textbf{Subcase II. Coordinate $i$ is unpaired in $\mathcal G$.}

This possibility occurs only when $n$ is odd. In this case, the $i$-th coordinate of the members of $\mathcal F^i$ ranges over all possible values in $V_0$. So the replacement family 
\[\mathcal F^i
=
\{\iota_i(s,u):s\in V_0,\ u\in T_e^i\},
\qquad \text{ and }
M=r|T_e^i|.\]

The family $\{\iota_i(t,v):t\in V_1,\ v\in T_o^i\}
\cup\{\mathbf 0\}
$
is Eventown. 
This implies
\[r|T_o^i|+1\le M=r|T_e^i|,
 \text{ in particular, } |T_o^i|<|T_e^i|.\]

Suppose that $T_o^i\ne\varnothing$, then $b_i>0$. Now 
\[
r|T_e^i|
=M
\le a_i|T_e^i|+b_i|T_o^i|
<(a_i+b_i)|T_e^i|.
\]
Since $a_i+b_i\le r+1$, we have
$a_i+b_i=r+1$.

The family
\[\{\iota_i(s,u):s\in V_0,\ u\in T_e^i\}
\cup
\{\iota_i(m,v):v\in T_o^i\}\]
is also an Eventown family of size
\[r|T_e^i|+|T_o^i|
=M+|T_o^i|
>M,\]
which contradicts the upper bound. 
Therefore $T_o^i=\varnothing$, and hence
\[\mathcal F^{i-1}
\subseteq
\{\iota_i(s,u):s\in V_0,\ u\in T_e^i\}
=
\mathcal F^i,\]

which implies $\mathcal F^{i-1} =\mathcal F^i= \mathcal{G}$.

\medskip
So starting with $\mathcal F^n=\mathcal G$ we have
\[\mathcal F
=
\mathcal F^0
=
\mathcal F^1
=
\cdots
=
\mathcal F^n
=
\mathcal G.\]
Since $\mathcal G$ has already been characterized in Lemma $\ref{lem:counting-G}$ and $\mathcal F=\mathcal G$, the desired characterization follows. This characterizes the maximum Eventown family of multisets.

\begin{acknowledgement}
The work of the author Anirban Banerjee is supported by the Department of Biotechnology, India (No. BT/PR40182/BTIS/137/88/2025).
  The work of the author Abisek Dewan is supported by University Grants Commission, India (Beneficiary Code/Flag: BWBDA00147662 U). The authors have used ChatGPT to assist with polishing arguments, paraphrasing language, and making grammatical corrections during the preparation of this manuscript.

\end{acknowledgement}

 \bibliographystyle{siam}
 	\bibliography{evenbib}
    \nocite{*}

\end{document}